\documentclass[11pt]{amsart}
\usepackage{amssymb}
\usepackage{graphics}
\usepackage{latexsym}
\usepackage{amsmath}
\usepackage{amssymb,amsthm,amsfonts}
\usepackage{diagrams}
\usepackage{amscd}
\usepackage[arrow, matrix, curve]{xy}
\usepackage{syntonly}

\title[Distribuation of CM points]{Distribuation of CM points of an infinite series of complete Calabi-Yau moduli spaces}
\author[Mao Sheng]{Mao Sheng}
\author[Jinxing Xu]{Jinxing Xu}
\email{msheng@ustc.edu.cn}\email{xujx02@ustc.edu.cn}
\address{School of Mathematical Sciences,
University of Science and Technology of China, Hefei, 230026, China}

\begin{document}

\theoremstyle{plain}
\newtheorem{thm}{Theorem}[section]
\newtheorem{theorem}[thm]{Theorem}
\newtheorem*{theorem*}{Theorem}
\newtheorem{lemma}[thm]{Lemma}
\newtheorem{corollary}[thm]{Corollary}
\newtheorem{proposition}[thm]{Proposition}
\newtheorem{addendum}[thm]{Addendum}
\newtheorem{variant}[thm]{Variant}
\theoremstyle{definition}
\newtheorem{construction}[thm]{Construction}
\newtheorem{notations}[thm]{Notations}
\newtheorem{question}[thm]{Question}
\newtheorem{problem}[thm]{Problem}
\newtheorem{remark}[thm]{Remark}
\newtheorem{remarks}[thm]{Remarks}
\newtheorem{definition}[thm]{Definition}
\newtheorem{claim}[thm]{Claim}
\newtheorem{assumption}[thm]{Assumption}
\newtheorem{assumptions}[thm]{Assumptions}
\newtheorem{properties}[thm]{Properties}
\newtheorem{example}[thm]{Example}
\newtheorem{conjecture}[thm]{Conjecture}
\newtheorem*{conjecture*}{Conjecture}
\numberwithin{equation}{thm}

\newcommand{\sA}{{\mathcal A}}
\newcommand{\sB}{{\mathcal B}}
\newcommand{\sC}{{\mathcal C}}
\newcommand{\sD}{{\mathcal D}}
\newcommand{\sE}{{\mathcal E}}
\newcommand{\sF}{{\mathcal F}}
\newcommand{\sG}{{\mathcal G}}
\newcommand{\sH}{{\mathcal H}}
\newcommand{\sI}{{\mathcal I}}
\newcommand{\sJ}{{\mathcal J}}
\newcommand{\sK}{{\mathcal K}}
\newcommand{\sL}{{\mathcal L}}
\newcommand{\sM}{{\mathcal M}}
\newcommand{\sN}{{\mathcal N}}
\newcommand{\sO}{{\mathcal O}}
\newcommand{\sP}{{\mathcal P}}
\newcommand{\sQ}{{\mathcal Q}}
\newcommand{\sR}{{\mathcal R}}
\newcommand{\sS}{{\mathcal S}}
\newcommand{\sT}{{\mathcal T}}
\newcommand{\sU}{{\mathcal U}}
\newcommand{\sV}{{\mathcal V}}
\newcommand{\sW}{{\mathcal W}}
\newcommand{\sX}{{\mathcal X}}
\newcommand{\sY}{{\mathcal Y}}
\newcommand{\sZ}{{\mathcal Z}}
\newcommand{\A}{{\mathbb A}}
\newcommand{\B}{{\mathbb B}}
\newcommand{\C}{{\mathbb C}}
\newcommand{\D}{{\mathbb D}}
\newcommand{\E}{{\mathbb E}}
\newcommand{\F}{{\mathbb F}}
\newcommand{\G}{{\mathbb G}}
\newcommand{\HH}{{\mathbb H}}
\newcommand{\I}{{\mathbb I}}
\newcommand{\J}{{\mathbb J}}
\renewcommand{\L}{{\mathbb L}}
\newcommand{\M}{{\mathbb M}}
\newcommand{\N}{{\mathbb N}}
\renewcommand{\P}{{\mathbb P}}
\newcommand{\Q}{{\mathbb Q}}
\newcommand{\R}{{\mathbb R}}
\newcommand{\SSS}{{\mathbb S}}
\newcommand{\T}{{\mathbb T}}
\newcommand{\U}{{\mathbb U}}
\newcommand{\V}{{\mathbb V}}
\newcommand{\W}{{\mathbb W}}
\newcommand{\X}{{\mathbb X}}
\newcommand{\Y}{{\mathbb Y}}
\newcommand{\Z}{{\mathbb Z}}
\newcommand{\id}{{\rm id}}
\newcommand{\rank}{{\rm rank}}
\newcommand{\END}{{\mathbb E}{\rm nd}}
\newcommand{\End}{{\rm End}}
\newcommand{\Hom}{{\rm Hom}}
\newcommand{\Hg}{{\rm Hg}}
\newcommand{\tr}{{\rm tr}}
\newcommand{\Sl}{{\rm Sl}}
\newcommand{\Gl}{{\rm Gl}}
\newcommand{\Cor}{{\rm Cor}}
\newcommand{\Aut}{\mathrm{Aut}}
\newcommand{\Sym}{\mathrm{Sym}}
\newcommand{\ModuliCY}{\mathfrak{M}_{CY}}
\newcommand{\HyperCY}{\mathfrak{H}_{CY}}
\newcommand{\ModuliAR}{\mathfrak{M}_{AR}}
\newcommand{\Modulione}{\mathfrak{M}_{1,n+3}}
\newcommand{\Modulin}{\mathfrak{M}_{n,n+3}}
\newcommand{\Gal}{\mathrm{Gal}}
\newcommand{\Spec}{\mathrm{Spec}}
\newcommand{\Jac}{\mathrm{Jac}}

\newcommand{\modulinm}{\mathfrak{M}_{AR}}

\newcommand{\mc}{\mathfrak{M}_{1,6}}

\newcommand{\mx}{\mathfrak{M}_{3,6}}

\newcommand{\tmc}{\widetilde{\mathfrak{M}}_{1,6}}

\newcommand{\tmx}{\widetilde{\mathfrak{M}}_{3,6}}

\newcommand{\smc}{\mathfrak{M}^s_{1,6}}

\newcommand{\smx}{\mathfrak{M}^s_{3,6}}

\newcommand{\tsmc}{\widetilde{\mathfrak{M}}^s_{1,6}}

\newcommand{\tsmx}{\widetilde{\mathfrak{M}}^s_{3,6}}

\newcommand{\proofend}{\hspace*{13cm} $\square$ \\}

\thanks{This work is supported by Chinese Universities Scientific Fund (CUSF), Anhui Initiative in Quantum Information Technologies (AHY150200) and National Natural Science Foundation of China (Grant No. 11622109, No. 11721101).}

\maketitle

\begin{abstract}
In the infinite series of complete families of Calabi-Yau manifolds $\tilde{f}_n: \tilde{\mathcal{X}}_n\rightarrow \mathfrak{M}_{n, n+3}$, where $n$ is an odd number, arising from cyclic covers of $\P^n$ branching along hyperplane arrangements (\cite{SXZ13}), the set of CM points is dense for $n=1, 3$ and finite for $n\geq 5$.
\end{abstract}


\section{Introduction}
\label{sec:introduction}
Recall that to a $\Q$-Hodge structure ($\Q$-HS) $h: \C^*\to \mathrm{GL}_{\R}(V\otimes_{\Q}\R)$, one attaches a $\Q$-algebraic group $MT(V)$, the so-called Mumford-Tate group, which is defined to be the smallest $\Q$-algebraic subgroup of $\mathrm{GL}_{\Q}(V)$ whose real points contain $h(\C^*)$. A basic result about Mumford-Tate group is that it is equal to the maximal $\Q$-algebraic subgroup of $\mathrm{GL}_{\Q}(V)$ which fixes all Hodge cycles occurring in $\bigoplus_{m,n}V^{\otimes m}\otimes V^{*\otimes n}$. A $\Q$-HS is said to be \emph{CM} if its attached Mumford-Tate group is commutative. Thus, the set of CM $\Q$-HSs makes a special kind of Hodge structures.

Indeed, a weight one $\Q$-polarized Hodge structure is of CM if and only if its corresponding Abelian variety, defined up to isogeny, has complex multiplication. It is well known that in the moduli space $\sA_g$ of $g$-dimensional (principally) polarized Abelian varieties, the set of CM points (a CM point is such a moduli point that represents an Abelian variety having complex multiplication) is Zariski dense. The same property holds for any Shimura subvariety in $\sA_g$. The celebrated Andr\'{e}-Oort conjecture for the $\sA_g$ case, which was recently solved by Tsimermann in \cite{Tsimerman}, asserts the converse statement: any irreducible component of the Zariski closure of a subset of CM points in $\sA_g$ is a Shimura subvariety. In general, we say a complex smooth  projective variety $X$ has CM if
 the $\Q$-HS on the middle cohomology $H^n(\mathcal{X}_s, \Q)$ is CM. Let $S$ be a smooth connected quasi-projective variety over $\C$ and $f: \mathcal{X}\rightarrow S$ be a smooth projective family,
we say a point $s\in S$ is a CM point if the fiber $\mathcal{X}_s$ has CM. In view of the Andr\'{e}-Oort conjecture,
we are naturally led to studying the distribution of CM points in a smooth projective family.
As a vast generalization of the Andr\'{e}-Oort conjecture, Bruno Klingler proposed
a general conjecture about the distribution of CM points in a $\Z$-variation of mixed
Hodge structures (Conjecture 5.3 in \cite{Klingler}). To our purpose, we state
Klingler's conjecture in the following form.
\begin{conjecture*}[Klingler]\label{conj:Klingler}
Let $S$ be a smooth connected quasi-projective variety over $\C$ and $f: \mathcal{X}\rightarrow S$ be a smooth projective family of relative dimension $n$. If $S$ contains a Zariski dense set of CM points then the period map
of the $\Z$-variation of
Hodge structures on the
 middle cohomology group $R^nf_*\Z_{\mathcal{X}}$ factors through
a Shimura variety.
\end{conjecture*}

Our aim is to examine Conjecture for an infinite series of families of Calabi-Yau (CY) manifolds, initially studied in our previous work \cite{SXZ13}. For each odd number $n$, let $r=\frac{n+3}{2}$, and let $\mathfrak{M}_{n, n+3}$ be the moduli space of $n+3$ hyperplanes arrangements of $\P^n$ in general position. In \cite{SXZ13}, we investigated several basic properties of universal families $f_n: \mathcal{X}_n\rightarrow \mathfrak{M}_{n, n+3}$ of $r$-fold cyclic covers of $\P^n$ branched along $n+3$ hyperplanes in general position. Note that the $n=1$ case is nothing but the classical universal family of elliptic curves. Among other results, we showed in loc. cit. that the family $f_n$ admits a simultaneous crepant resolution $\tilde{\mathcal{X}}_n\rightarrow \mathcal{X}_n$; the resulting CY family $\tilde{f}_n:\tilde{\mathcal{X}}_n\rightarrow \mathfrak{M}_{n, n+3}$ is actually complete and versal at each point; they all have Yukawa coupling length one. In this note, we are going to show the following
\begin{theorem*}\label{thm introduction: finite CM}
The set of CM points of the family $\tilde{f}_n$ is Zariski dense for $n=1,3$ and finite for $n\geq 5$.
\end{theorem*}
It was early shown in our work \cite{SXZ15} that the period map of $\tilde f_n$ does not factor through a Shimura variety if $n\geq 5$ (see Theorem 6.7 \cite{SXZ15}). Therefore, the above theorem provides an evidence to Conjecture. Note that it remains an interesting open problem whether the one-dimensional mirror quintic CY family contains only finitely many CM points or not. By a result of Deligne \cite{Deligne} on the monodromy group of the mirror quintic family, it should be the case after Conjecture. As far as we know, our result provides the first example of complete CY families with finite CM points.

\section{Families from hyperplane arrangements}\label{sec:Families from hyperplane arrangements}
Given a hyperplane arrangement $\mathfrak{A}$ in $\P^n$ in general position, the cyclic cover of $\P^n$ branched along $\mathfrak{A}$ is an interesting algebraic variety. When the hyperplane arrangement $\mathfrak{A}$ moves in the coarse moduli space of hyperplane arrangements, we get a family of projective varieties. In this section, we collect some known facts about Hodge structures of these cyclic covers. Detailed proofs of all of the statements in this section can be found in \cite{SXZ13}.

We say an ordered arrangement $\mathfrak{A}=(H_1,\cdots, H_m)$ of hyperplanes in $\P^n$ is in general position if no $n+1$ of the hyperplanes intersect in a point, or equivalently, if the divisor $\sum_{i=1}^mH_i$ has simple normal crossings.

Given an odd number $n$, and let $r=\frac{n+3}{2}$. For each ordered hyperplane arrangement $(H_1,\cdots, H_{n+3})$ in $\P^n$ in general position, we can define a (unique up to isomorphism) degree $r$ cyclic cover of $\P^n$ branched along the divisor $\sum_{i=1}^{n+3}H_i$. In this way, if we denote  the coarse moduli space of   ordered $n+3$ hyperplane arrangements  in $\P^n$ in general position by $\mathfrak{M}_{n, n+3}$, then we obtain a universal family $f_n:\mathcal{X}_{n}\rightarrow \mathfrak{M}_{n, n+3}$ of degree $r$ cyclic covers of $\P^n$ branched along $n+3$ hyperplane arrangements in general position. In  \cite{SXZ13}, we constructed a simultaneous crepant resolution $\pi: \tilde{\mathcal{X}}_{n}\rightarrow \mathcal{X}_{n}$ for the family $f$ without changing the middle cohomology of fibers. Moreover, this simultaneous crepant resolution gives an $n$-dimensional projective  Calabi-Yau family which  is maximal in the sense that its Kodaira-Spencer map is an isomorphism at each point of $\mathfrak{M}_{n, n+3}$. We denote this smooth projective Calabi-Yau family by $\tilde{f}_n:\tilde{\mathcal{X}}_{n}\rightarrow \mathfrak{M}_{n, n+3}$.

Now we recall the relation between a cyclic cover of $\P^1$ branched along points and that of $\P^n$ branched along hyperplane arrangements.

Suppose $(p_1,\cdots, p_{n+3})$ is a collection of $n+3$ distinct points on $\P^1$, and put $H_i=\{p_i\}\times \P^1\times \cdots \times \P^1$. By the natural identification between $\P^n$ and the symmetric power $Sym^n(\P^1)$ of $\P^1$, we can view each $H_i$ as a hyperplane in $\P^n$.  Then it can be shown that $(H_1,\cdots, H_{n+3})$ is a hyperplane arrangement in $\P^n$ in general position. A direct computation shows that this construction gives an isomorphism between the moduli spaces $\mathfrak{M}_{1, n+3}$ and $\mathfrak{M}_{n, n+3}$. Moreover, for $r=\frac{n+3}{2}$, if we denote $C$ as the $r$-fold cyclic cover of $\P^1$ branched along the $n+3$ points $(p_1,\cdots, p_{n+3})$, and $X$ as the $r$-fold cyclic cover of $\P^n$ branched along the corresponding hyperplane arrangement $(H_1,\cdots, H_{n+3})$, then we have an isomorphism
\begin{equation}\label{equation:relation between X and C}
X \simeq C^n/S_n\ltimes N.
\end{equation}
Here $N$ is the kernel of the summation homomorphism $(\Z/r\Z)^n \rightarrow \Z/r\Z$. The action of $\Z/r\Z$  on $C$ is induced from the cyclic cover structure, and $S_n$ acts on $C^n$ by permutating the $n$ factors.

From the isomorphism (\ref{equation:relation between X and C}),
we see the natural $\Q$-mixed Hodge structure on the middle cohomology group $H^n(X, \Q)$ is pure. Moreover, since the  simultaneous crepant resolution $\tilde{\mathcal{X}}_{n}\rightarrow \mathcal{X}_{n}$ of the universal family $\mathcal{X}_{n}\xrightarrow{f_n} \mathfrak{M}_{n,n+3}$ does not change the middle cohomologies of the fibers, the two $\Q$-VHS (rational variation of  Hodge structures) $R^n\tilde{f}_{n*}\Q_{\tilde{\mathcal{X}}_{n}}$ and $R^n f_{n*}\Q_{\mathcal{X}_{n}}$ are isomorphic.

\section{Proof of the main result}

For  the smooth projective family of Calabi-Yau $n$-folds $\tilde{f}_n:\tilde{\mathcal{X}}_{AR}\rightarrow \mathfrak{M}_{n, n+3}$, we have the following proposition:

\begin{proposition}\label{prop:finite CM}
The family $\tilde{f}_3$ has a dense set of fibers that has CM. If $n\geq 5$ is an odd number, then up to isomorphism, there exist at most  finitely many fibers of $\tilde{f}_n$ that  has CM.
\end{proposition}

\begin{proof}
First note that since the two $\Q$-VHS $R^n\tilde{f}_{n*}\Q_{\tilde{\mathcal{X}}_{n}}$ and $R^n f_{n*}\Q_{\mathcal{X}_{n}}$ are isomorphic, we only need to prove this theorem for the family $f_n$.
By the discussions in Section \ref{sec:Families from hyperplane arrangements}, for each fiber $X$ of $f_n$,  we can find $n+3$ distinct point $p_1, \cdots, p_{n+3}$ on $\P^1$, such that
\begin{equation}\label{equ:cor between X and C in general}
X\simeq C^n/S_n\ltimes N
\end{equation}
Here $C$ is the  $r$-fold cyclic cover of $\P^1$ branched along $p_1, \cdots, p_{n+3}$, the group $N$ is the kernel of the summation homomorphism $\oplus_{i=1}^n\Z/ r\Z\xrightarrow{\sum}\Z/r\Z$, and $S_n$ is the permutation group of $n$ elements. The action of $S_n$ on the product $C^n$ is by permutating the factors, and  the action of $N$ is induced by the cyclic cover action of $\Z/r\Z$ on $C$.

By the isomorphism \eqref{equ:cor between X and C in general}, we get an isomorphism of cohomology groups:
\begin{equation}\label{equ:cor between cohomology groups}
H^n(X, \ \Q)\simeq H^n(C^n, \ \Q)^{S_n\ltimes N}
\end{equation}
where $H^n(C^n, \ \Q)^{S_n\ltimes N}$ means the $S_n\ltimes N$-fixed part of $H^n(C^n, \ \Q)$. Let $W=H^1(C, \ \Q)$, and $V=H^n(X, \ \Q)$. As a subgroup of $\oplus_{i=1}^{n} \Z/ r\Z$, the group $N$ acts naturally on the tensor product $W^{\otimes n}$. We can also describe the  action of $S_n$ on $W^{\otimes n}$:
$$
\sigma \cdot \alpha_1\otimes \alpha_2\otimes \cdots \otimes \alpha_n=(-1)^{sgn (\sigma)}\cdot \alpha_{\sigma^{-1}(1)}\otimes \alpha_{\sigma^{-1}(2)}\otimes \cdots \otimes \alpha_{\sigma^{-1}(n)}
$$
for  $\sigma\in S_n$ and $\alpha_i\in W$ , $1\leq i\leq n$. Here $sgn(\sigma)=\pm 1$ means the signal of the permutation $\sigma$. By K\"unnenth formula, the isomorphism \eqref{equ:cor between cohomology groups} implies the following isomorphism:
\begin{equation}\label{equ:cor between W and V}
V\simeq (W^{\otimes n})^{S_n\ltimes N}
\end{equation}

Note $V$ admits a $\Q$-Hodge structure of  weight $n$, and $W$ admits a weight one $\Q$-Hodge structure. Moreover, the isomorphism \eqref{equ:cor between W and V} is an isomorphism between $\Q$-Hodge structures of weight $n$, with the Hodge structure on $(W^{\otimes n})^{S_n\ltimes N}$ induced from the Hodge structure on $W$. A direct computation shows that the genus $g$ of $C$ satisfies the following relation:
$$
2g=(n+1)(r-1)
$$
Recall $r=\frac{n+3}{2}$. So if $n\geq 5$, then $g\geq 9$. Note the main result in \cite{CLZ} asserts that  if $g\geq 8$, then up to isomorphism, there exist at most  finitely many superelliptic curves of genus $g$ with CM Jacobian. Then the $n\geq 5$ cases follows from this result and  the following claim. For the  $n=3$ case, the
corresponding family of curves is the universal family of cyclic triple
 covers of $\P^1$ branched along six distinct points, and by \cite{DM} this
 curve family is a Shimura family. It is well-known that the CM points
 are dense in a Shimura family. Then the $n=3$ case also follows from the following claim.

\textbf{Claim}: $V$ \textit{ has CM if and only if } $W$ \textit{ has CM}.
\end{proof}

\textbf{Proof of Claim}:
We first define the $\Q$-subgroup $G$  of $GL(W)$ consisting of elements commutating with the action of $\Z/r\Z$ on $W$. i.e.,
$$
G:=\{g\in GL(V)| g\cdot \alpha=\alpha \cdot g, \ \forall \alpha \in \Z/r\Z\}
$$
Then $G$ acts on $V=(W^{\otimes n})^{S_n\ltimes N}$ by:
$$
g\cdot \alpha_1\otimes \alpha_2\otimes \cdots \otimes \alpha_n= g\cdot \alpha_1\otimes g\cdot \alpha_2\otimes \cdots \otimes g\cdot \alpha_n
$$
for $g\in G$ and $\alpha_1\otimes \alpha_2\otimes \cdots \otimes \alpha_n\in (W^{\otimes n})^{S_n\ltimes N}$.
In this way, we get a homomorphism between $\Q$-algebraic groups:
$$
\varphi: G\rightarrow GL(V).
$$
By considering the complex points, we can verify directly  that the kernel $K$ of $\varphi$ is a diagonalizable group lying in the center of $G$. We also have the following commutative diagram:
\begin{diagram}\label{diag:MT group}
\C^*&  \rTo^{h_W} & G(\R)\\
& \rdTo_{h_V} &\dTo_{\varphi}  \\
& &GL(V_{\R})
\end{diagram}
Here $h_V$ and $h_W$ are the  morphisms associated with the Hodge structures on $V$ and $W$ respectively.

Note the Mumford-Tate group $MT(W)$ is contained in $G$, and by  the commutative diagram above, we get $MT(V)\subset \varphi(MT(W))$. Hence $W$ has CM implies $V$ has CM.

Note also  the following exact sequence:
$$
1\rightarrow K \rightarrow \varphi^{-1}(MT(V))\rightarrow MT(V)
$$
By this exact sequence, if $V$ has CM, then $\varphi^{-1}(MT(V))$ is a diagonalizable group whose real points contain $h_W(\C^*)$, and hence $MT(W)\subset \varphi^{-1}(MT(V))$ is a torus, and $W$ has CM. This finishes the proof of Claim.     \hfill $\square$

As a corollary, we can deduce  our main result.

\begin{theorem}
If $n=3$, then the set of CM points of the family $\tilde{f}_n$ is Zariski dense in $\mathfrak{M}_{n, n+3}$. If $n\geq 5$ is an odd number, then the set of CM points of the family $\tilde{f}_n$ is finite.
\end{theorem}
\begin{proof}
We only need to prove the $n\geq 5$ cases. Let $\mathcal{M}_n$ be the coarse moduli space of the fibers of $\tilde{f}_n$,
then the family $\tilde{f}_n$ induces a morphism of quasi-projective  varieties
$\varphi: \mathfrak{M}_{n, n+3}\rightarrow \mathcal{M}_n$. Let
$S$ be the set of CM points in $\mathfrak{M}_{n, n+3}$. Then Proposition
\ref{prop:finite CM} implies that the image $\varphi(S)$ is a finite set in
$\mathcal{M}_n$. On the other hand, since the Kodaira-Spencer map of $\tilde{f}_n$
is an isomorphism at each point of  $\mathfrak{M}_{n, n+3}$, we know each fiber of
$\varphi$ is a zero-dimensional close subvariety of $\mathfrak{M}_{n, n+3}$.
Then $S=\varphi^{-1}(\varphi(S))$ is a zero-dimensional close subvariety of $\mathfrak{M}_{n, n+3}$.
So $S$ is a finite set.

\end{proof}

\textbf{Acknowledgements}
We thank Professor Kang Zuo for several enlightening conversations related to this work.  A primitive form of Conjecture \ref{conj:Klingler} for Calabi-Yau varieties has already appeared in the work \cite{Zhang} of Yi Zhang, who left us on April 2nd, 2019. The first named author is very grateful to Yi Zhang for his generous help and warm encouragement during different periods of the career. The sudden pass-away of Yi Zhang is a great loss to his family, teachers and friends.  

\end{document}